\documentclass[a4paper,11pt,twoside]{article}

\usepackage[body={175mm,250mm}]{geometry}
\usepackage{a1}
\usepackage[english]{babel}
\usepackage[latin1]{inputenc}
\usepackage{amsfonts,amssymb}
\usepackage{amsmath}
\usepackage{graphicx}	

\usepackage{makeidx}
\makeindex

\def\mapright#1#2#3{\smash{\mathop{\hbox to
#3{\rightarrowfill}}\limits^{#1}_{#2}}}

\def\mapleft#1#2#3{\smash{\mathop{\hbox to
#3{\leftarrowfill}}\limits^{#1}_{#2}}}

\def\mapright#1#2{\smash{\mathop{\hbox to 0.90cm{\rightarrowfill}}\limits^{#1}_{#2}}}
\def\mapleft#1#2{\smash{\mathop{\hbox to 0.90cm{\leftarrowfill}}\limits^{#1}_{#2}}}

\def\mapleftright#1#2{\smash{\mathop{\hbox to 0.80cm{\leftarrowfill \rightarrowfill}}\limits^{#1}_{#2}}}

\title{A graphical calculus for tangles in surfaces
}
\author{Peter M.~Johnson \and S\'ostenes Lins}
\date{}

\begin{document}

\maketitle

\begin{abstract}
We show how the theory of Reidemeister moves of tangles, in a surface with boundary,
can be simplified by concentrating on we call fine tangles, where the surface has been
cut into very simple pieces and the moves are of a restricted kind.
This reworking of the graphical foundations for link and tangle theory can be expected to have
a variety of applications, including ones involving 3-manifolds.  It opens the way to new
approaches for defining `facial' state-sum invariants, based partly on assigning symbols
to faces of tangles.
\end{abstract}

\section{Introduction}

Our aim is to make a substantial improvement in a part of the foundations of combinatorial
topology often used when defining invariants of knots and links, or of related
objects such as diagrams (isotopy classes) of tangles in surfaces.
The usual requirements for defining invariants of tangles in surfaces are weakened
by restricting the allowed Reidemeister moves of type 2 and,
to compensate, considering only fine tangles, as defined below.
To evaluate a new invariant on an arbitrary tangle, the tangle should first
be prepared, via unrestricted moves of type 2, to make it fine.
In a companion article \cite{joli2012B}, this idea is combined with ones of
Reshetikhin and Turaev to define new families of `facial' state-sum invariants.
For these, a much coarser process of refinement suffices, producing what we call
well-placed tangles, as there is a way to
deal directly with faces that have handles or disconnected boundaries,
so that they do not need to be cut into finer pieces.
The present article deals only with the graphical foundations needed to justify the
general approach, which applies even to unoriented surfaces.
As a corollary to our main result relating tangles with fine tangles, we give
the weaker but useful result for well-placed tangles.
It remains to be seen what other approaches can exploit
fine tangle theory to define new types of invariants.

First some historical background will be sketched.  The Reidemeister moves
for links (including knots), presented as generic plane projections
of links undergoing isotopy in space, were discovered independently by Reidemeister
\cite{reid1926} and Alexander and Briggs \cite{alexbriggs1927}.  They became
widely known through Reidemeister's second book on knot theory \cite{reid1932}.
The focus is not on a single link or a diagram $L$ but on the class
$[L]$ of diagrams equivalent to the given one via appropriate local
modifications.  Among them we include moves of type 0 (isotopies of the plane,
affecting the tangle but not its diagram), type 1 (often restricted or avoided),
type 2 (to be restricted below, and subdivided into cases), and type 3. 
With minor variations, most notably versions where link components are oriented,
Reidemeister moves usually provide the foundation used to obtain relations
whose satisfaction is necessary and sufficient for defining invariants
of links and related objects such as framed tangles in surfaces.

An avenue that opens further possibilities for obtaining invariants
is to calculate only after obtaining some kind of prepared diagram.
Our main idea is to use a new sort of preparation, described in the next section,
but here we mention some other approaches and give details that will be useful.

One of the most fruitful ideas is to present links as braids, in the
plane or even within 3-manifolds, as in Lambropolou and Rourke \cite{laro1997}. 
The extra structure can then be exploited.  Among other approaches, an especially
deep one, due mainly to Thurston (see for example \cite{thurston1997three}), involves
triangulating link complements after finding a hyperbolic structure, and applies to
all links in space except those of a restricted kind.
Another set of ideas concern framing.
There were several independent discoveries concerning twist and writhe, starting with
C\u{a}lug\u{a}reanu \cite{ca1959} in 1959, but these long went unperceived by knot theorists,
whose concern was combinatorial topology rather than differential geometry.
As late as 1983, Trace \cite{tr1983} showed how suitable preparation of planar diagrams
of knots or links makes the first Reidemeister move, on adding and removing curls,
completely unnecessary.
The only Reidemeister moves that matter for defining invariants of oriented links in the plane
are those of type 2 and 3, as each link component can be prepared by adding curls to 
adjust the writhe and rotation number, say to 0 and 1.  In practice, it proved better to
ignore the rotation number, to facilitate the study of more general objects such as links
in oriented surfaces.
The focus shifted to framed objects, with diagrams that admit a writhe-preserving form of the
first Reidemeister move, also known as a ribbon move.
Unframed links, invariant even under moves of type 1,
can be regarded as 0-framed links: one need only prepare all diagrams
by adding curls to make all link components have writhe 0.

We are indebted to
the Departamento de Matem\'atica, UFPE, Brazil and
to the Centro de Inform\'atica, UFPE, Brazil for financial support.
The second author is
also supported by a research grant from CNPq/Brazil, proc. 301233/2009-8.

\section{Definitions and results}

We will present a new way to prepare tangle diagrams
and will describe which moves are appropriate for defining the notion of
equivalence between prepared diagrams.
We usually refer to tangles but implicitly work only up to isotopy. 
Following and adapting usual conventions and definitions for PL objects, tangles
(including links) will be drawn in a compact surface $S$ whose boundary $\partial S$
can be regarded as consisting of $k$ disjoint `circles', also called holes, left after
removing open disks from a surface without boundary.
In $S$, a {\em tangle} consists of a finite set
of curves, where each intersection is transverse and endowed
with under-over crossing information, such that each curve is closed or has
endpoints in $\partial S$, and is otherwise disjoint from $\partial S$.
Tangles are assumed to be unoriented, merely to avoid having to 
list oriented versions of the basic moves below.
Surfaces can be treated abstractly, as in combinatorial topology, but it is also convenient to
use informal metric language, for example mentioning two close and almost parallel segments in $S$.

We prefer not to mention thickened surfaces, so faces of a tangle $T$
in $S$ can be defined to be the connected components of the complement of
$T \cup \partial S$ in $S$.
In $S$, the boundary of a face relative to $T$ is contained in $T \cup \partial S$,
which decomposes naturally into arcs, used in the usual way to
codify combinatorial information.
This describes how to glue faces at arcs
to reconstruct up to homeomorphism the tangle in the surface.
If one prefers each face to have a boundary in which no two arcs are glued together,
consider tangles that have been thickened so that their strands become thin bands in $S$.

Equivalence of tangles $T$ in $S$ will be defined via Reidemeister moves of types 0
(homeomorphisms of $S$ isotopic to the identity), 2 and 3, giving a version of the regular
isotopy classes of Kauffman \cite{kauffman1990invariant}.
Optionally, one could also allow some sort of type 1 move, usually the ribbon move (see for
example Fig.~5 of \cite{joli2012B}), which gives the important class of framed tangles.
A Reidemeister move not of type 0 modifies only part of a tangle within some open disk.

Tangles (especially links) studied here differ from virtual ones in two ways: it is not
permitted to modify $S$ by adding or removing handles at parts of $S$ disjoint from the tangle,
and tangles in the same orbit of the mapping class group of $S$ are not necessarily
equivalent.  The last condition opens the way to defining invariants that in some way
use homotopy or homology classes, as in \cite{bafe2001}.
A useful alternative is to consider the surface $S'$ obtained by removing a point from
a connected $S$, so that a tangle in $S'$ will have a distinguished `infinite' region.
It can be obtained from the plane by removing open disks, then
attaching handles and, for $S'$ unorientable, a crosscap.

Our main contribution is to remove an obstacle that until now has blocked progress with
state-sum invariants that depend in part on assigning symbols to faces.
The problem involves a non-local feature of Reidemeister moves of type 2.
When two intersecting
strands pull apart to form two disjoint strands, three regions (viewed locally) become one.
One of these, within the disk, disappears, and the other two fuse.
These two may be the same region, globally, but if not the move is said to be
{\em admissible}.  As an alternative to admissible moves, one can use two special kinds
of move that are undeniably local:
those of types 2E and 2F (E for eight, F for finger), that we now define.
Each comes in two forms, with more refined names as shown in the
the next figures.

\bigskip
\begin{figure}[!ht]
\begin{center}
\includegraphics[width=10cm]{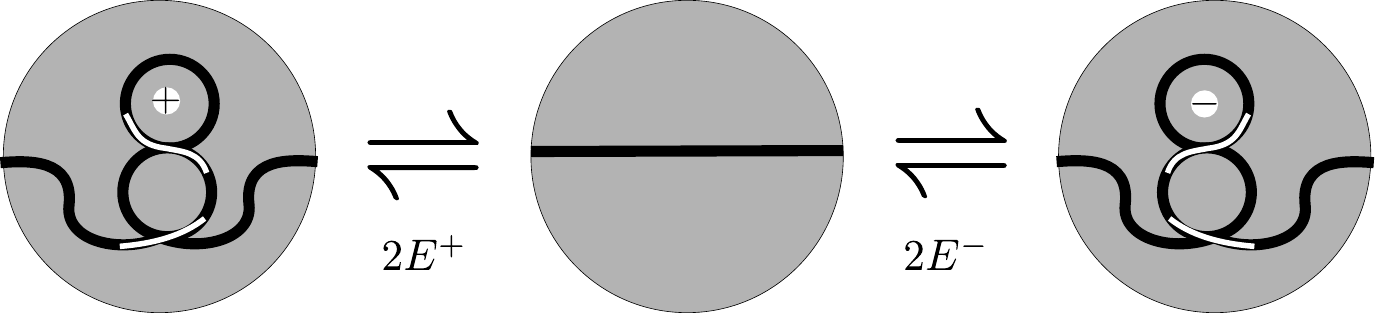} \\
\caption{\sf {\bf The 8-moves $2E^+$ and $2E^-$}
}
\label{fig:eigthmoves}
\end{center}
\end{figure}

\begin{figure}[!ht]
\begin{center}
\includegraphics[width=12cm]{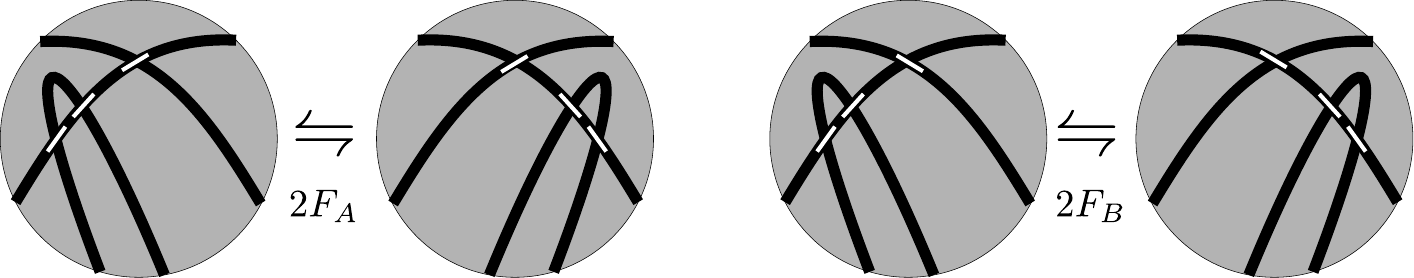} \\
\caption{\sf {\bf The finger-rotation moves $2F_A$ and $2F_B$}
}
\label{fig:fingerrotation}
\end{center}
\end{figure}

\goodbreak 
\bigskip

\begin{definition}
A tangle $T$ in $S$ is {\em fine} if each face $f$ is homeomorphic to an open disk
or an annulus (or cylinder) whose boundary $\partial f$ in $S$ intersects $\partial S$ in a
connected set, which is a boundary component of $S$ if $f$ is an annulus.
Two such tangles in $S$ are {\em strongly equivalent} if one can be obtained
from the other by a sequence of moves of types 0, 2E, 2F and 3.
\end{definition}

Local moves will turn out to be sufficient when working with tangles that
are fine.  The illustrations below
show part of a fine tangle in a surface $S$ near a face
having one edge (the maximum allowed) in $\partial S$, that edge forming
a semicircle around a hole of $S$ shown in white.
The central illustration indicates where admissible moves can be made, while
the others can be obtained from their neighbours via moves of type 2E (in
two cases) or 2F.
A barrier is shown where a move between one pair of neighbours is
impeded by the presence of an edge in $\partial S$.
The crossing type where strands cross is displayed only when relevant.

\begin{figure}[!ht]
\begin{center}
\includegraphics[width=10cm]{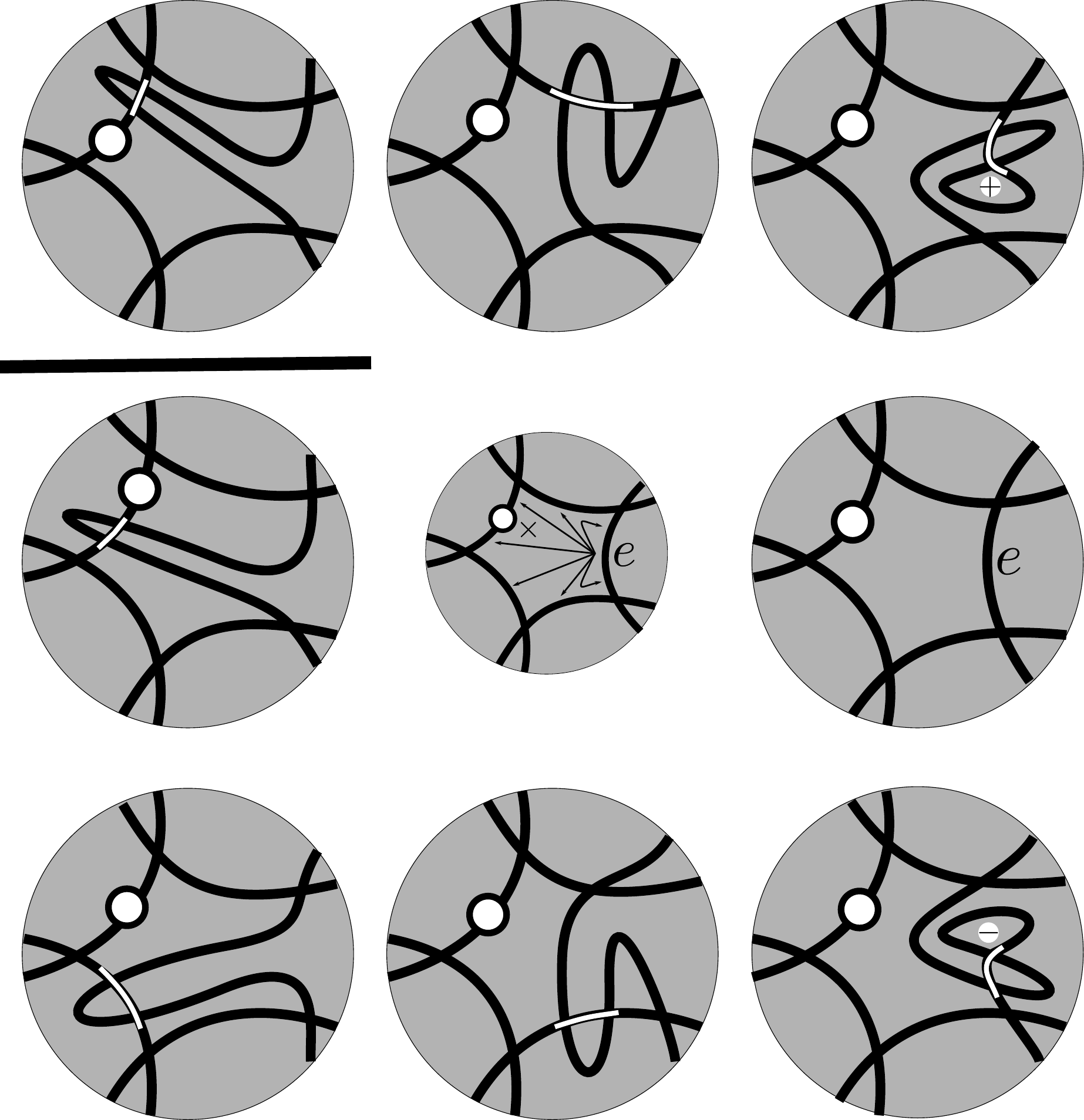} \\
\caption{\sf {\bf Factoring admissible type 2 moves into 2E and 2F moves}
}
\label{fig:factoringmoves2}
\end{center}
\end{figure}

\begin{theorem}
Any move of type 2, applied to a fine tangle 
in the direction that creates a pair of crossings,
produces another such tangle.  The move can be factored into
a move of type 2E followed by a sequence of moves of type 2F.
\end{theorem}

\begin{proof}
Such a move can be regarded as taking place
essentially within a face $f$, via a finger-creation move where one edge of
the face grows within $f$ into a thin finger which then
passes under an edge (possibly the one where it originated).
The faces adjacent with $f$ need not be different from $f$, as some
pairs of edges around the boundary of $f$ could be identified, but any diagram
will at least show faithfully the arrangement of edges around each vertex.
All new faces created are topological disks, and the conditions
referring to boundaries that intersect $\partial S$ clearly continue to hold.
In the typical case shown above, by moving around the outer illustrations,
starting with the copy of the central one, any other can be reached.
This is done by making a move of type 2E, then moves of type 2F, always within
a topological disk.
In cases such as the one shown, where a barrier occurs, only one
direction to proceed is possible.  The only other case of interest is when the
region where the move is made (creating crossings) is annular.  By definition, this
region must have one boundary component in $\partial S$ and the other disjoint
from $\partial S$, and again there is only one direction in which to proceed.
\end{proof}

Reidemeister moves of type 2, in the direction that creates a new
pair of crossings, can be used to convert tangles, or their diagrams, into equivalent
ones that are fine.  It is this process that we call preparation.
Sequences of such moves give what will be called finger-creation moves, where
a a small segment of $T$ is replaced by a thin finger (almost parallel pair
of strands) that expands within $S$, avoiding $\partial S$, and passes under
$T$, other fingers, or even earlier parts of the current finger.

\begin{proposition}
Let $T$ be a tangle in a surface $S$ such that every connected component of $S$ contains
points of $T$.  Then $T$ can be prepared by applying finger-creation moves repeatedly
so that the new tangle is fine.
\end{proposition}

\begin{proof}
Given $T$ as above, the procedure is straightforward.
One can cut faces until their boundary is connected and
intersects $\partial S$ in at most one connected piece.
A way to remove handles is shown in the figure below. 
Any unorientable faces at this stage are
M\"obius strips, and can easily be cut into disks.
When a face has a boundary component fully contained in $\partial S$,
finger-creation moves cannot approach that component, and the best that
can be done is to isolate the component, creating an annular face.
\end{proof}

\begin{figure}[!ht]
\begin{center}
\includegraphics[width=14cm]{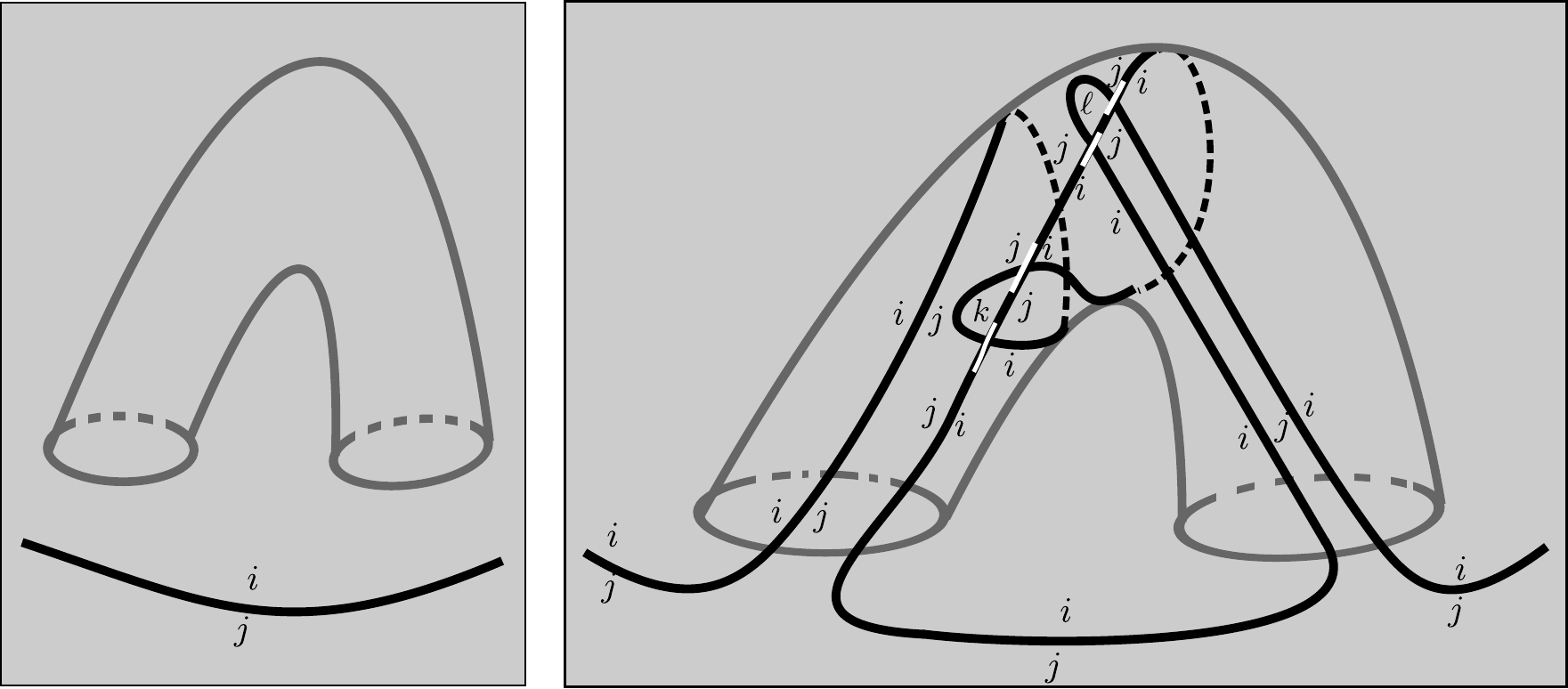} \\
\caption{\sf {\bf Cutting a handle by finger-creation moves}
}
\label{fig:fingerrefinement}
\end{center}
\end{figure}

\begin{theorem}
All prepared tangles obtained from the same tangle are strongly equivalent.
\end{theorem}

\begin{proof}
Consider two preparations $T_1$ and $T_2$ of the same tangle $T$.
It can be assumed that the small strands of $T$ that were initially expanded
are mutually disjoint.
Form a new tangle $T_3$ using both preparations, where any new intersections
between $T_1$ and $T_2$ are adjusted to be transverse, with the part belonging to
$T_2$ passing under that of $T_1$.
One can then convert both $T_1$ and $T_2$ to $T_3$ via a series of moves of types
2E and 2F between fine tangles, so $T_1$ and $T_2$ are strongly equivalent.
\end{proof}

This will now be strengthened.  We treat only regular isotopy, defined via moves of
types 0, 2, 3.   There are analogous versions which are immediate consequences.

\begin{theorem}
All prepared tangles from regularly isotopic tangles are strongly equivalent.
\end{theorem}

\begin{proof}
Suppose $T'$ (resp.\ $U'$) is an arbitrary tangle obtained by preparing $T$ (resp. $U$)
in $S$, and let $T=T_0, \dots, T_n=U$ be a sequence of tangles in which
adjacent ones are related by Reidemeister moves of types 0, 2, and 3.
We first show how to obtain a sequence $T'=T'_0, \dots, T'_n$ of fine tangles,
where each $T'_i$ is prepared from $T_i$ by using new strands
that always pass under the ones of $T_i$, and
consecutive tangles $T'_i$, $T'_{i+1}$ are related by a series of restricted
Reidemeister moves.  After the given initial $T'$, each $T'_{i+1}$ in the sequence
is obtained from $T'_i$ by carrying out moves related to the one between $T_i$ and
$T_{i+1}$, after moving out of the way any new strands
that obstruct the move (passing under the others).
All this can clearly be done using only moves of types 0, 2E, 2F, and 3.
Finally, note that $T'_n$ and $U'$ are preparations of the same tangle $U$.
Thus $T'$ and $U'$ are strongly equivalent.
\end{proof}

As an obvious reformulation of the above ideas, we have:

\begin{corollary}
Each regular isotopy class $[T]$ of tangles in $S$ determines, via preparations,
a unique strong equivalence class $\langle T'\rangle$ of tangles in $S$,
whose members are the fine tangles contained in $[T]$.
\end{corollary}

The class $\langle T'\rangle$ can be regarded as a graphical invariant of the original
regular isotopy class $[T]$.  This provides a more promising starting point for methods
that construct various families of algebraically-defined invariants,
as only certain kinds of Reidemeister moves need be respected.
In a companion article \cite{joli2012B}, the first in a planned series,
we introduce invariants that use the new idea.  In our setup, the requirement
for using fine tangles is far more stringent than necessary, as it is possible
to work directly with tangles satisfying the following mild restriction.

\begin{definition}
A tangle $T$ in $S$ is {\em well-placed} if, for each of its faces $f$, the
boundary $\partial f$ intersects $\partial S$ in a connected set.
\end{definition}

The notion of equivalence between such tangles is again given by
sequences of Reidemeister moves, now involving only tangles that are well-placed.
There is also an appropriate version of the previous corollary.

\begin{corollary}
Each regular isotopy class $[T]$ of tangles in $S$ determines, via preparations,
a unique equivalence class $\{ T'\}$ of well-placed tangles in $S$,
whose members are the well-placed tangles contained in $[T]$.
\end{corollary}

\begin{proof}
The preparation method for obtaining fine tangles via moves of type 2,
when applied to a well-placed tangle, produces a sequence of tangles
that are clearly well-placed.  Given two well-placed tangles in $[T]$,
the last theorem, or its first corollary, can now be applied.
\end{proof}

\bibliographystyle{plain}
\bibliography{bibtexIndex.bib}

\vspace{10mm}
\begin{center}

\begin{tabular}{l}
   Peter M. Johnson\\
   Departamento de Matem\'atica, UFPE\\
   Recife--PE \\
   Brazil\\
   peterj@dmat.ufpe.br
\end{tabular}
\hspace{7mm}
\begin{tabular}{l}
   S\'ostenes Lins\\
   Centro de Inform\'atica, UFPE \\
   Recife--PE \\
   Brazil\\
   sostenes@cin.ufpe.br
\end{tabular}
\end{center}

\end{document}